\documentclass[12pt]{article}
\usepackage{amsthm}
\usepackage{amssymb,latexsym,amsmath}
\oddsidemargin=\evensidemargin \footskip=36pt \voffset=-1.5cm
\def\dim{\rm dim}

\def\C{{\mathbb C}}

\def\ker{{\rm { Ker}}}

\def\Irr{{\rm { Irr}}}

\def\ooo{{\rm O}^\vartheta{}}

\def\s{\sigma}

\def\bull{\vrule height 1.2ex width 1ex depth -.1ex }

\makeatletter
\renewcommand{\subsection}{\@startsection{subsection}{2}{0mm}{7mm}{4mm}
{\bf\normalsize}}

\def\nmrt{\begin{enumerate}}
\def\enmrt{\end{enumerate}}
\def\tm#1{\item[{\rm (#1)}]}

\makeatother
\newtheorem{formula}{}[section]

\newtheorem{definition}[formula]{Definition}
\newtheorem{corollary}[formula]{Corollary}
\newtheorem{remark}[formula]{Remark}
\newtheorem{lemma}[formula]{Lemma}
\newtheorem{theorem}[formula]{Theorem}

\newtheorem{example}[formula]{Example}

\begin{document}

\title{Association schemes with at most two nonlinear irreducible characters and applications to finite groups
}
\author{
Javad Bagherian  \\
Department of Mathematics, University of Isfahan,\\ Isfahan
81746-73441, Iran.\\bagherian@sci.ui.ac.ir}

\maketitle
\begin{abstract}
An irreducible character $\chi$ of an association scheme is called nonlinear
if the multiplicity of $\chi$ is greater than $1$.
The main result of this paper gives a characterization of commutative association schemes with
at most two nonlinear irreducible characters. This yields a characterization of finite groups
with at most two nonlinear irreducible characters. A class of noncommutative
association schemes with at most two nonlinear irreducible character
is also given.
\end{abstract}
\smallskip


\section{Introduction}
In the character theory of association schemes, the character values of irreducible characters
 give many useful information about association schemes. In particular,
the multiplicities of irreducible characters paly an important role for  determining the structure of
association schemes. For instance, see \cite{hir}. 

An irreducible character $\chi$ of an association scheme $(X,S)$ is called nonlinear
if the multiplicity of $\chi$ is greater than $1$.
An interesting problem in the character theory of association schemes is
what can be said about $(X,S)$ when the number of nonlinear irreducible characters is known.
In particular, if $(X,S)$ is a group association scheme induced from a finite group $G$, then since
the number of nonlinear irreducible characters of $G$ and $(X,S)$ is equal,  any characterization of association schemes with
a given number of nonlinear irreducible characters yields a characterization of finite groups in terms of the number of
nonlinear irreducible characters. It should be mentioned that there are some characterizations of
finite groups with one nonlinear irreducible character; for example see \cite{one}.

In this paper we characterize the structure of commutative association schemes with at most two
nonlinear irreducible characters.
This yields  
a characterization of finite groups with at most two nonlinear
irreducible characters. Moreover, we give a characterization of noncommutative association schemes
with one nonlinear irreducible character. Finally, a class of noncommutative association schemes with two
nonlinear irreducible characters is given.

\section{Preliminaries}
Let us first state some necessary definitions and notation.
For details, we refer the reader to  \cite{Zig} for the background of association
schemes. Throughout this paper,  $\C$ denotes the complex numbers.

\subsection{Association schemes}

\begin{definition}
Let $X$ be a finite set and $S=\{s_0, s_1, \ldots, s_n\}$ be a partition of $X \times X$. Then $(X,S)$ is
called an association scheme with $n$ classes if the following properties hold:
\nmrt \tm{i}
$s_0=\{(x,x) | x\in X\}$.
\tm{ii} For every $s\in S$, $s^*$ is also in $S$, where $s^*:=\{(x,y) |(y,x) \in s\}$.
\tm{iii}
For every $g,h,k \in S$, there exists a nonnegative integer $\lambda_{ghk}$ such that for every $(x,y)\in k$,
there exist exactly $\lambda_{ghk}$ elements $z\in X$ with $(x,z)\in g$ and $(z,y)\in h$.
\enmrt
\end{definition}
The diagonal relation $s_0$ will be denoted by $1$.
For each $s\in S$, we call $n_s=\lambda_{ss^* 1}$ the {\it valency} of $s$. For any nonempty subset $H$ of $S$, put
$n_H =\sum_{h\in H}n_h$. We call $n_S$ the {\it order} of $(X,S)$.
An association scheme $(X,S)$ is called {\it commutative} if for all $g, h, k\in S $, $\lambda_{ghk}=\lambda_{hgk}$.
Let $H$ and $K$ be nonempty subsets of $S$. We define $HK$ to be the set of all elements $t\in S$ such that there
exist element $h\in H$ and $k\in K$ with $\lambda_{hkt}\neq 0$. The set $HK$ is called the {\it complex product} of
$H$ and $K$. If one of factors in a complex product consists of a single element $s$, then one usually writes $s$
for $\{s\}$.
A nonempty subset $H$ of $S$ is called a {\it closed subset} if $HH\subseteq H$.
A closed subset $H$ of $S$ is called {\it strongly normal}, denoted by $H\vartriangleleft^\sharp S$,
if $sHs^*  = H$ for any $s\in S$.
We put $\ooo(S)=\bigcap_{H\vartriangleleft^\sharp S}H$ and call it the {\it thin residue} of $H$.
For a closed subset $H$ of $S$, the number $n_S/n_H$ is called the {\it index } of $H$ in $S$ and is denoted by $|S:H|$.

Let $(X,S)$ be a scheme and $H$ be a closed subset of $S$. For every $x\in X$ put $xH=\bigcup_{h\in H} xh$, where
$xh=\{y\in X|(x,y)\in h\}$. A {\it subscheme} $(X,S)_{xH}$ is a scheme with the points $xH$ and the set of relations
$\{s_{xH}|s\in H\}$, where $s_{xH}=s\bigcap (xH\times xH)$.

Moreover, if we put $X/\!\!/H=\{xH|x\in X\}$ and $S/\!\!/H=\{s^H|s\in S\}$ where  $s^H=\{(xH,yH)|y\in xHsH\}$,
then $(X/\!\!/H,S/\!\!/H)$ is a scheme, called the {\it quotient scheme} of $(X, S)$ over $H$.
Note that a
closed subset $H$ is strongly normal if and only if the quotient scheme $(X/\!\!/H,S/\!\!/H)$
is a group with respect to the relational product if and only if $ss^*\subseteq H$, for every $s\in S$ (see \cite[Theorem 2.2.3]{Zig}).
In particular, since $S/\!\!/\ooo(S)$ is a finite group, we can consider the derived
subgroup $(S/\!\!/\ooo(S))'$ of $S/\!\!/\ooo(S)$. Suppose  $S'$ be the inverse image of $(S/\!\!/\ooo(S))'$.
Then the quotient scheme $(X/\!\!/S',S/\!\!/S')$ is an abelian group with respect to the relational product.

\subsection{Characters of association schemes}

Let $(X,S)$ be an association scheme. For every $s\in S$, let $\s_s$ be the adjacency matrix of $s$.
For any nonempty subset $H$ of $S$, we put $\s_H:=\sum_{ h\in H}\s_h$.
For convenience, $\s_{1}$ is denoted by $1$.
It is known that $\mathbb{C}S=\bigoplus_{s\in S}\mathbb{C}\s_s$,
the {\it adjacency algebra} of $(X,S)$, is a semisimple algebra (see \cite[Theorem 4.1.3]{Zig}).
The set of irreducible characters of $S$  is denoted by $\Irr(S)$.
 One can see that $1_S\in \rm{Hom_\C(\C S,\C)}$
such that $1_S(\s_s)=n_s$ is an irreducible character of $\C S$, called the
{\it principal} character. In \cite{hanakirep}, Hanaki has shown that the
irreducible characters of $S/\!\!/\ooo(S)$ can be considered as irreducible characters of $S$.

Let $\Gamma_S$ be a representation of $\C S$ which sends $\s_s$ to itself for every $s\in S$. Let $\gamma_S$ be the
character afforded by $\Gamma_S$. Then one can see that  $\gamma_S(1)=|X|$ and $\gamma_S(\s_s)=0$ for
every $1 \neq s\in S$. Consider the following irreducible decomposition of $\gamma_S$,
\begin{eqnarray}\label{3}\nonumber
\gamma_S=\displaystyle\sum_{\chi \in \Irr(S)} m_\chi \chi.
\end{eqnarray}
We call $m_\chi$ the {\it multiplicity} of $\chi$. 
One can see that $m_{1_S}=1$ and $|X|=\sum_{\chi\in \Irr(S)}m_\chi\chi(1)$ (see \cite[section 4]{Zig}).
For every $\chi\in \Irr(S)$, put $$\ker(\chi)=\{s\in S\mid \chi(\s_s)=n_s\chi(1)\}.$$ Then $\ker(\chi)$ is a normal closed subset of $S$
and it follows from \cite[Theorem 5.3]{hanakirep} that for every normal closed subset $H$ of $S$, $$\Irr(S/\!\!/H)=\{\chi\in \Irr(S)\mid H\subseteq \ker(\chi)\}.$$

\begin{example}\label{exam}
Let $G$ be a finite group and $C_0=\{1\}, C_1, \ldots, C_h $ be the conjugacy classes of $G$. Define
$R_i$ by $(x,y)\in R_i$ if and only if $xy^{-1}\in C_i$, where $x,y\in G$. Put $S=\{R_i\}_{0\leq i \leq h}$. Then
$(G,S)$ is an association scheme, which is called the {\it group association scheme} of $G$.
For every relation $R_i$ of $S$, $n_{R_i}=|C_i|$ and one can see that $R_i\in \ooo(S)$ if and only if $C_i\subseteq G'$, where $G'$ is the derived subgroup of $G$.

The adjacency algebra of $(G,S)$ is isomorphic to the algebra ${\rm Z}(\C G)$ with the basis ${\rm Cla}(G)$,
where ${\rm Z}(\C G)$ is the center of group algebra $\C G$ and ${\rm Cla}(G)=\{K_0, \ldots, K_h\}$,
with $K_i=\sum_{g\in C_i}g$. Furthermore, $\{\omega_{\chi}|~\chi\in \Irr(G)\}$ is
the set of irreducible characters of $Z(\C G)$, where
\begin{eqnarray}\nonumber\label{e}
\omega_{\chi}(K_i)=\frac{\chi(g)|C_i|}{\chi(1)},
\end{eqnarray}
for some $g\in C_i$. Since $|G|=\sum_{\chi\in \Irr(G)}\chi(1)^2$, it follows that
$m_{\omega_\chi}=\chi(1)^2$, for every $\chi\in \Irr(G)$.
\end{example}

Let $(X,S)$ be an association scheme  and $T$ be a closed subset of $S$. Suppose that $L$ is a $\C T$-module which affords
the character $\varphi$, and $V$ is a $\C S$-module which affords the character $\chi$.
Then $V$ is a $\C T$-module which affords the restriction $\chi_T$ of $\chi$ to $\C T$, and
$L^S=L \otimes_{\C T} \C S$ is a $\C S$-module which affords the induction $\varphi^S$ of $\varphi$.
Suppose that $T$ is  strongly normal. Put $G=S/\!\!/T$. Let $\varphi$ be an irreducible character
of $T$ and $L$ be an irreducible  $\C T$ module affording $\varphi$. Consider the induction of $L$ to $S$.
Then
\begin{eqnarray}\nonumber
L^S=L\otimes_{\C T} \C S=\bigoplus_{s^T\in S/\!\!/T}L\otimes\C (TsT).
\end{eqnarray}
The {\it stabilizer} $G\{L\}$ of $L$ in $G$ is defined by
$$
G\{L\}=\{s^T\in S/\!\!/T|L\otimes\C (TsT) \cong L\}.
$$
One can see that $G\{L\}$ is a subgroup of $G$. 
The set of $S/\!\!/T$-conjugates of $L$ is $\{L\otimes\C (TsT)| s\in S,~ L\otimes\C (TsT)\neq 0\}$. From \cite{clifford2} it
follows that if $L$ and $L'$ are $S/\!\!/T$-conjugates, then $L^S\cong {L'}^S$.

\begin{theorem}(See \cite{clifford2}.){\label{cli}}
Let $(X,S)$ be a scheme and $T$ be a strongly closed subset of $S$.
Put $G=S/\!\!/T$. Then for every $\chi\in \Irr(S)$, there exists a positive integer $e$ such that
$$\chi_{T}=e\displaystyle\sum^n_{i=1}\varphi_i,$$
where $\varphi_i, 1\leq i\leq n$, are $S/\!\!/{T}$-conjugate irreducible characters of $T$.
\end{theorem}

\begin{theorem}(See \cite{hanaki2}.)\label{pindex}
Let $(X,S)$ be an association scheme and $T$ be a strongly normal closed subset of $S$.
Suppose that $G=S/\!\!/T$ is the cyclic group of prime order $p$.
Suppose that $\Irr(G)=\{\zeta_i|1\leq i\leq p \}$.
Then for $\chi\in \Irr(S)$, one of the following statements holds:
\nmrt
\tm{1}
$\chi_T\in \Irr(T)$ and $(\chi_T)^S=\sum_{i=1}^p \chi\zeta_i$,
\tm{2}
$\chi(\s_s)=0$, for any $s\in S\setminus T$ and $\chi_T$ is a sum of at most $p$ distinct irreducible characters. If $\psi$ is an
irreducible constituent of $\chi_T$, then $\psi^S=\chi$.
\enmrt
\end{theorem}
The product of characters in association schemes has been given in \cite{hana3} by Hanaki. If $(X,S)$ is an association scheme and $T$
is a strongly normal closed subset of $S$, then it follows from \cite[Theorem 3.3]{hana3}  that for every $\chi\in \Irr(S)$  and $\zeta\in \Irr(S/\!\!/T)$, the character product $\chi\zeta$ defined by
$$\chi\zeta(\s_s)=\chi(\s_s)\zeta(\s_{s^T})$$
is a character of $S$. If $\zeta(1)=1$, then $\chi\zeta \in \Irr(S)$. So  $\Irr(S/\!\!/S')$ acts on $\Irr(S)$ by the above multiplication.

\subsection{products of association schemes}
The wedge product of association schemes is a way to construct a new association scheme from old
ones and has been given by Muzychuk in \cite{we}. A special case of the wedge product of association schemes is the wreath product.
We refer the reader to \cite{we} for more details.
Here we give the definition of wedge product of association schemes. This is equivalent to Muzychuk definition of wedge product.

Let $(X,S)$ be an association scheme and $K\subseteq  H$ be closed subsets of $S$ such that
\nmrt
\tm{a} $\sigma_{K}\sigma_s=n_K \s_s=\s_s \s_K$, for every $s\in S\setminus H$;
\tm{b} $K\unlhd S$.
\enmrt
Then $S$ is called the wedge product of association schemes $(X,S)_{xH}$ and $(X/\!\!/K, S/\!\!/K)$ for some $x\in X$.
\\

In the above definition if $K=H$, then
$S$ is called the {\it wreath product} of association schemes $(X,S)_{xH}$ and $(X/\!\!/H, S/\!\!/H)$ for some $x\in X$.

The following result is immediately obtained from the definition of wreath product.

\begin{lemma}\label{pedrs}
Let $(X,S)$ be a commutative association scheme and $H$ be a closed subset of $S$. Then the following are equivalent:
\nmrt
\tm{1} $S$ is the wreath product of association schemes $(X,S)_{xH}$ and $(X/\!\!/H, S/\!\!/H)$,
\tm{2} $|S|=|H|+|S/\!\!/H|-1$,
\tm{3} For every $h\in H$ and $s\in S\setminus H$, $\s_s\s_h=\s_h\s_s=n_h\s_s$.
\enmrt
\end{lemma}


The following easy lemma is useful.
\begin{lemma}\label{val}
Let $(X,S)$ be a commutative association scheme and $K\subseteq H$ be closed subsets of $S$. Let
$S$ be the wedge product of association schemes $(X,S)_{xH}$ and $(X/\!\!/K, S/\!\!/K)$.
If $K$ is strongly normal in $S$, then $n_s=n_K$ for every $s\in S\setminus H$.
\end{lemma}
\begin{proof}
For every $s\in S \setminus H$ we have $$n_{s^K}n_K=n_{Ks},$$ (see \cite[Theorem 1.5.4(v)]{Zig}).
Since  $n_{s^K}=1$ it follows that $n_K=n_{Ks}$. On the other hand,   $Ks=\{s\}$. Hence $n_s=n_K$.
\end{proof}\hfill\bull

\subsection{Group-like schemes}
Let $(X, S)$ be an association scheme. We define a binary relation $\sim$ on $S$ as follows. For $s,t\in S$, we write
$s\sim t$ if
\begin{eqnarray}\label{eqn}
\chi(\s_s)/n_s=\chi(\s_t)/n_t,
\end{eqnarray}
 for every $\chi\in \Irr(S)$. Then $\sim$ is an equivalence relation.
For $s\in S$, put $\widetilde{s}=\bigcup_{t\sim s}s$ and $\widetilde{S}=\{\widetilde{s}|s\in S\}$.
If $Z(\C S)=\bigoplus_{\widetilde{s}\in \widetilde{S}}\C\s_{\widetilde{s}}$, then $(X,S)$ is called a {\it group-like scheme}.
If $(X, S)$ is group-like, then $(X, \widetilde{S})$ becomes a commutative association scheme.

\begin{theorem}(See \cite[Theorem 4.1]{hanaki}.)\label{glike2}
For an association scheme $(X, S)$, the following statements are equivalent:
\nmrt
\tm{1} $(X,S)$ is a group-like scheme,
\tm{2} $\dim_\C Z(\C S)=|\widetilde{S}|$,
\tm{3} for every $\chi, \psi\in \Irr(S)$, $\chi\psi$ is a linear combination of $\Irr(S)$, where $$\chi\psi(\s_s)=\frac{1}{n_s}\chi(\s_s)\psi(\s_s),~~~~\forall s\in S.$$

\enmrt
\end{theorem}

The following easy lemma is useful.

\begin{lemma}\label{glike}
An association scheme $(X,S)$ is group-like  if and only if for
every $\chi, \psi \in \Irr(S)$ and every $s, h\in S$, $\chi\psi(\s_s\s_h)=\chi\psi(\s_h\s_s)$.

\end{lemma}

\begin{proof}
For every $\chi, \psi\in \Irr(S)$, consider the linear function $\chi\psi: \C S\rightarrow \C$ by
$$\chi\psi(\displaystyle\sum_{s\in S}\lambda_s\s_s)=\displaystyle\sum_{s\in S}\lambda_s\chi\psi(\s_s).$$
It follows from \cite[Theorem 4.2]{hana3} that $\chi\psi$ is a linear combination of irreducible characters
if and only if $\chi\psi(\s_s\s_h)=\chi\psi(\s_h\s_s)$ for every $s, h\in S$. The result now follows from
Theorem \ref{glike2}.

\end{proof}\hfill\bull

\section{Main Results}
Let $G$ be a finite group. An irreducible character $\chi$ of $G$ is called nonlinear if $\chi(1)>1$.
In this section we first define the concept of a nonlinear  irreducible character
for association schemes and then give a characterization of association schemes with at most two nonlinear irreducible characters.

Let $(X,S)$ be an association scheme. We say that an irreducible character $\chi$ of $S$ is {\it linear} if $m_\chi=1$;
otherwise $\chi$ is called {\it nonlinear}. It follows from
\cite[Lemma 2.4(v)]{hir} that $\chi\in \Irr(S/\!\!/S')$ if and only if $m_\chi=1$. So
$|S:S'|$ is the number of linear characters of $S$ and
$\Irr(S)\setminus\Irr(S/\!\!/{S'})$ is the set of nonlinear  characters of $S$.
In particular, if $(X,S)$ is commutative, then an irreducible character
$\chi\in\Irr(S)$  is linear if and only if  $\chi\in \Irr(S/\!\!/\ooo(S))$ and so $\Irr(S)\setminus\Irr(S/\!\!/\ooo(S))$
is the set of nonlinear  irreducible characters of $S$.

If $(G,S)$ is the group association scheme of $G$, then an irreducible character $\chi$ of $G$ is nonlinear if and only if the irreducible
character $\omega_\chi$ of $S$ is nonlinear; see Example \ref{exam}.

\subsection{Association schemes with one nonlinear irreducible character }
In this section we give a characterization of association schemes with exactly one nonlinear irreducible character.

\begin{lemma}\label{one}
A commutative association scheme $(X,S)$ has exactly one nonlinear irreducible character if and only if
$|\ooo(S)|=2$ and $S$ is the wreath product of association schemes $(X,S)_{x\ooo(S)}$ and  $(X/\!\!/\ooo(S),S/\!\!/\ooo(S))$ for $x\in X$.
\end{lemma}
\begin{proof}
First we assume that $S$ contains exactly one nonlinear irreducible character. Then  $|S|=|S/\!\!/\ooo(S)|+1$. Put
$T=\ooo(S)$. Suppose that $S/\!\!/T=\{1^T, s^T_1, \ldots, s^T_n\}$, for some relations $s_i\in S$.
Since $|S|=|S/\!\!/T|+1$, it follows that $S=\{1,t,s_1, \ldots, s_n\}$ and  $T=\{1,t\}$, for some relation $t$.
Now since $$|S|=|S/\!\!/T|+1=|T|+|S/\!\!/T|-1,$$
it follows from Lemma \ref{pedrs} that $S$ is the wreath product of association schemes $(X,S)_{xT}$ and $(X/\!\!/T, S/\!\!/T)$
for some $x\in X$.

Conversely, it follows from \cite[Theorem 4.1]{irrwr} that $S$ contains exactly one nonlinear irreducible character.
\end{proof}\hfill\bull

From the above lemma we can give the following characterization of finite groups with exactly one nonlinear irreducible
character.

\begin{corollary}(See \cite{one}.)\label{con}
Let $G$ be a finite group. Then the following are equivalent:
\nmrt
\tm{1} $G$ has exactly one nonlinear irreducible character,
\tm{2}$G'$ is the union of two conjugacy classes of $G$ and for every $g\in G \setminus G'$,
coset $gG'$ is the conjugacy class of $G$ containing $g$,
\tm{3} $G$ is an extra-special $2$-group, or $G$ is a doubly transitive Frobenius group with a cyclic Frobenius complement and
Frobenius kernel  $G'$ which is an elementary abelian $p$-group.
\enmrt
\end{corollary}
\begin{proof}
$(1)\Rightarrow (2)$ 
Since $G$ has exactly one nonlinear irreducible character,  $(G,S)$ also has one nonlinear irreducible character.
It  follows from Lemma \ref{one} that $|\ooo(S)|=2$ and $S$ is the wreath product of association schemes $(G,S)_{g\ooo(S)}$ and  $(G/\!\!/\ooo(S),S/\!\!/\ooo(S))$ for $g\in G$. This implies that $G'$ contains two conjugacy classes $C_0$ and $C_i$ for some $1\leq i\leq d$
and Lemma \ref{pedrs} shows that $K_iK_j=K_jK_i=|C_i| K_j$ for every $j\neq 0,i$. So we conclude that $C_iC_j \subset C_j$ for every
$j\neq 0, i$. Hence for every $g\in C_j$, $gC_i\subset C_j$. Now from Lemma \ref{val} we have $|C_j|=|G'|=|C_i|+1$  and so $gG'=C_j$.

$(2)\Rightarrow (1)$ Consider the group association scheme $(G,S)$.
Since $G'$  contains two conjugacy classes $C_0$ and $C$, we have $|\ooo(S)|=2$. Moreover, since
for every conjugacy class $C' \not\in \{C_0, C\}$ and $g\in C'$, $C'=gG'$  we get
$C'C\subset C'$, and so $K'K=KK'=|C|K'$. Then it follows from
Lemma \ref{pedrs} that $S$ is the wreath product of association schemes $(G,S)_{g\ooo(S)}$ and
$(X/\!\!/\ooo(S),S/\!\!/\ooo(S))$ for $g\in G$. So Lemma \ref{one} shows that $(G,S)$ contains exactly one
nonlinear irreducible character and hance $G$ also has exactly one nonlinear irreducible character.

$(2)\Rightarrow (3)$ Suppose that $G'$ is the union of two conjugacy classes  $C_0$ and $ C_1$,
and for every $g\in G \setminus G'$, coset $gG'$ is
the conjugacy class of $G$ containing $g$.
Then $G'$ is a $p$-group and so $G$ is solvable. Moreover, since for every conjugacy class $C$ and $g\in C$ we have
$gG'=C$ it follows that $G'$ is the unique minimal normal subgroup of $G$. So it follows from \cite[Lemma 12.3]{is} that
all nonlinear irreducible characters of $G$ have
equal degree $f$ and  one of the following holds:
\nmrt
\tm{a}$G$ is a $p$-group, $Z(G)$ is cyclic and $G/Z(G)$ is elementary abelian,
\tm{b}$G$ is a Frobenius group with an abelian Frobenius complement $H$ of
order $f$. Also, $G'$ is the Frobenius kernel and is an elementary abelian $p$-group.
\enmrt
If $(a)$ holds, then $|C_1|=1$ and $|G'|=2$ . Since $G'\subseteq Z(G)$ and $|C|>1$, for every conjugacy
class $C\not\in \{C_0,C_1\}$, we have $G'=Z(G)$ and $G/G'$ is an elementary abelian $2$-group.
Hence $G$ is an extra-special $2$-group, as desired.

Now suppose $(b)$ holds.  Let $|G'|=p^n$. Since $|G/G'|=|H|=f$, we have $|G|=fp^n.$ On the other hand,
since statements $(1)$ and $(2)$
are equivalent, $G$ has one nonlinear irreducible character and so
$$|G|=\sum_{\chi\in \Irr(G)}{\chi(1)}^2=f^2+|G/G'|.$$ Then $f=p^n-1$ and $G$ is a  Frobenius group
of order $(p^n - 1)p^n$. Since $H$ is abelian and every Sylow subgroup of $H$ is cyclic or a generalized
quaternion group we conclude that $H$ is cyclic.

$(3)\Rightarrow (2)$ Suppose that  $G$ is an extra-special $2$-group, or $G$ is a doubly transitive Frobenius group with a
cyclic Frobenius complement and  Frobenius kernel $G'$ which is an elementary abelian $p$-group.
Then in both cases $G$ has exactly one nonlinear irreducible character; see \cite{hup}.
Since the statements $(1)$ and $(2)$ are equivalent, it follows that $G'$ is the union of  two conjugacy classes and
for every $g\in G\setminus G'$, coset $gG'$ is the
conjugacy class of $G$ containing $g$.

\end{proof} \hfill\bull

\begin{theorem}
An association scheme $(X,S)$ has exactly one nonlinear irreducible character if and only if it is a group-like scheme,
$|\widetilde{S'}|=2$ and $\widetilde{S}$ is the wreath product of association schemes $(X,\widetilde{S})_{x\widetilde{S'}}$ and  $(X/\!\!/\widetilde{S'},\widetilde{S}/\!\!/\widetilde{S'})$ for $x\in X$.
\end{theorem}
\begin{proof}
Suppose that $(X,S)$ has exactly one nonlinear irreducible character. First note that for
every $s,t \in S\setminus S'$,  $s^{S'}=t^{S'}$ if and only if
for every $\psi\in \Irr(S/\!\!/{S'})$, $$\psi(\s_s)/n_s=\psi(\s_{s^{S'}})=\psi(\s_{t^{S'}})=\psi(\s_t)/n_t$$
(see \cite[Theorem 3.5]{hanakirep}). On the other hand,
since $\chi$ is the only irreducible character of $S$ with $m_\chi>1$
we conclude that the orbit of $\chi$ has length $1$ under the action of $\Irr(S/\!\!/S')$ on $\Irr(S)$. This implies that for every  $\zeta\in \Irr(S/\!\!/S')$,
$\chi\zeta=\chi$. Since for every $s\in S\setminus S'$, there exists $\zeta\in \Irr(S/\!\!/S')$
such that $\zeta(\s_{s^{S'}})\neq 1$, from equality  $\chi(\s_s)=\chi(\s_s)\zeta(\s_{s^{S'}})$ we have $\chi(\s_s)=0$.
Then for every $s,t \in S\setminus S'$, $s^{S'}=t^{S'}$ if and only if
for every $\psi\in \Irr(S)$, $\psi(\s_s)/n_s=\psi(\s_t)/n_t$.

Now let $1\neq s\in S'$. Since $\displaystyle\sum_{\psi\in \Irr(S)}m_\psi \psi(\s_s)=0$,
it follows that  $$m_\chi \chi(\s_s)+\displaystyle\sum_{\chi\neq\psi\in \Irr(S)}m_\psi \psi(\s_s)=0.$$
Then
\begin{eqnarray}\nonumber
\chi(\s_s)=\frac{-|S/\!\!/{S'}|n_s}{m_\chi}
\end{eqnarray}
and hence
\begin{eqnarray}\label{hhh}
\chi(\s_s)/n_s=\frac{-|S/\!\!/{S'}|}{m_\chi}.
\end{eqnarray}
So for nondiagonal relations  $s, t\in S'$,  (\ref{hhh})
shows that $$\frac{\chi(\s_t)}{n_t}=\frac{-|S/\!\!/{S'}|}{m_\chi}=\frac{\chi(\s_s)}{n_s}.$$
Thus $$\frac{\psi(\s_s)}{n_s}=\frac{\psi(\s_t)}{n_t},$$ for every $\psi\in \Irr(S)$.

Now let $\sim$ be the equivalence relation which is defined in (\ref{eqn}).
From the above we conclude that  $$\widetilde{S}=1\cup \widetilde{t} \cup \{s^{S'}| s\in S\setminus S'\},$$
for some $t\in S'$. Then $|\widetilde{S}|=|S/\!\!/{S'}|+1$. Since $\dim _\C Z(\mathbb{C}S)=|\Irr(S)|=|S/\!\!/{S'}|+1$  we have
$\C \widetilde{S}=Z(\mathbb{C}S)$ and it follows from Theorem \ref{glike2} that $(X,S)$ is a group-like scheme.
Since $(X,\widetilde{S})$ has exactly one nonlinear irreducible character,  Lemma \ref{one} shows that
$|\ooo(\widetilde{S})|=2$ and $\widetilde{S}$ is the wreath product of association schemes $(X,\widetilde{S})_{x\ooo(\widetilde{S})}$ and  $(X/\!\!/\ooo(\widetilde{S}),\widetilde{S}/\!\!/\ooo(\widetilde{S}))$ for $x\in X$. Since $\ooo(\widetilde{S})=\widetilde{S'}$,
the result follows.

Conversely, it follows from Lemma \ref{one} that
$(X,\widetilde{S})$ has exactly one nonlinear irreducible character. Since $m_{\widetilde{\chi}}=\chi(1) m_\chi$, for every
$\chi\in \Irr(S)$ we conclude that $(X,S)$ has exactly one nonlinear irreducible character. The proof is now complete.
\end{proof}\hfill\bull

\subsection{Association schemes with two nonlinear irreducible characters}

In this section we first give a characterization of commutative association schemes with
exactly two nonlinear irreducible characters. Then we give a class of association schemes with two
nonlinear irreducible characters.

\begin{theorem}\label{20}
Let $(X,S)$ be a commutative association scheme such that $|S|>3$.  Then $(X,S)$ contains exactly two nonlinear irreducible characters
if and only if  one of the following holds:
\nmrt
\tm{i}$|\ooo(S)|=3$ and $S$ is the wreath product of association schemes
$(X,S)_{x\ooo(S)}$ and $(X/\!\!/\ooo(S), S/\!\!/\ooo(S))$ for $x\in X$,
\tm{ii}$|\ooo(S)|=2$ and there exists a strongly normal closed subset $H$ of $S$ containing $\ooo(S)$ such that $|H|=4$,
$|S:H|=2$, and $S$ is the wedge product of $(X,S)_{xH}$ and $(X/\!\!/\ooo(S), S/\!\!/\ooo(S))$ for some $x\in X$.
\enmrt
\end{theorem}
\begin{proof}
Let $\chi$ and $ \psi$ be two nonlinear irreducible characters of $S$. Put $T=\ooo(S)$. We consider two cases.

First, suppose that $\chi_T\neq \psi_T$. Then   $\Irr(T)=\{1_T, \chi_T, \psi_T\}$.
Since  $\chi_T\neq \psi_T$, it follows that
the orbits of $\chi$ and $\psi$ have length 1 under the action of $\Irr(S/\!\!/T)$ on $\Irr(S)$.
So for every $\zeta\in \Irr(S/\!\!/T)$ we have $\chi\zeta=\chi$ and $\psi\zeta=\psi$.
Since for every $s\in S\setminus T$, there exists $\zeta\in \Irr(S/\!\!/T)$ such that $\zeta(\s_{s^{T}})\neq 1$,
from equalities $\chi(\s_s)\zeta(\s_{s^{T}})=\chi(\s_s)$ and $\psi(\s_s)\zeta(\s_{s^{T}})=\psi(\s_s)$,
we conclude that $\chi(\s_s)=\psi(\s_s)=0$.
Now we show that for every $s,t\in S\setminus T$, $s^T\neq t^T$ and hence
$|S/\!\!/T|=|S|-|T|+1$. Suppose that $s^T=t^T$. Then for every $\varphi\in \Irr(S/\!\!/T)$, we have
$$\varphi(\s_s)/n_s=\varphi(\s_{s^T})=\varphi(\s_{t^T})=\varphi(\s_t)/n_t.$$
Since $\chi(\s_s)=\psi(\s_t)=0$, it follows that for every $\varphi\in \Irr(S)$,  $\varphi(\s_s)/n_s=\varphi(\s_t)/n_t$.
This is a contradiction since the character table of $\C S$ is a nonsingular matrix. So $|S/\!\!/T|=|S|-|T|+1$.  This implies that  $|S|=|S/\!\!/T|+2=|S/\!\!/T|+|T|-1$
and Lemma \ref{pedrs} shows that $S$ is the wreath product of association schemes $(X,S)_{xT}$ and $(X/\!\!/T,S/\!\!/T)$.

Second, suppose that $\chi_T=\psi_T$. Then $\Irr(T)=\{1_T, \chi_T\}$ and so $|T|=2$.
Let $T=\{1,s\}$, for some $s\in T$. Then since $|S|=|S/\!\!/T|+2$, it follows that there are
exactly two relations $g, h\in S\setminus T $ such that
$g^T=h^T$. So $H=\{1, s,g,h\}$ is a closed subset of $S$ such that $h\in sg$. Moreover, since $g^T$ is an involution of $S/\!\!/T$ it follows that $H/\!\!/T$ is a subgroup of $S/\!\!/T$ and thus $H$ is a strongly normal closed subset of $S$ with $|H:T|=2$.

Now since $T\subseteq H$ and for every $s\in S\setminus H$, $sT=\{s\}$ it follows that
$S$ is the wedge product of $(X,S)_{xH}$ and $(X/\!\!/T, S/\!\!/T)$ for some $x\in X$.
\\

Conversely, if $(i)$ holds, then it follows from \cite[Theorem 4.1]{irrwr} that
$(X,S)$ contains exactly two nonlinear irreducible characters.

Now suppose that $(X,S)$ satisfies condition $(ii)$.
First we show that for every $\chi \in \Irr(S)\setminus \Irr(S/\!\!/\ooo(S))$ and $s\in S\setminus H$,
$\chi(\s_s)=0$. Since $\ooo(S)\nsubseteq \ker(\chi)$, there exists at least $t\in \ooo(S)$ such that $\chi(\s_t)\neq n_t$.
On the other hand, $\s_t\s_s=n_t\s_s$. So equality $\chi(\s_t)\chi(s_t)=n_t\chi(\s_s)$ shows that $\chi(\s_s)=0$.
Moreover, since $|H|=4$ and $|H:\ooo(S)|=2$, we
have exactly two irreducible characters $\lambda, \mu\in \Irr(H) \setminus \Irr(H/\!\!/\ooo(S))$. Let
$\chi$ and $\psi$ be irreducible characters of $S$ such that $(\chi_H, \lambda)\neq 0$ and $(\psi_H, \mu)\neq 0$.
Now consider the sequence
$$
H=H_0\subseteq H_1\subseteq \ldots \subseteq H_n=S
$$
of closed subsets of $S$ such that $H_i/\!\!/H_{i-1}$ is a group of prime order.
Since $\chi(\s_s)=\psi(\s_s)=0$ for every $s\in S\setminus H$,
it follows from  Theorem \ref{pindex} that $\lambda^{H_i}$  and $ \mu^{H_i}$ are irreducible characters of $H_i$.
Hence we conclude that $\chi=\lambda^S$ and $\psi=\mu^S$. Clearly, for every nonlinear irreducible character $\varphi$  of $S$,
we must have  $\varphi=\lambda^S$ or
$\varphi=\mu^S$. Thus  $(X,S)$  contains exactly two nonlinear irreducible characters.
The proof is now complete.
\end{proof}\hfill\bull

From the above theorem we can obtain the following characterization of finite groups with exactly two nonlinear irreducible
characters.

\begin{corollary}\label{corj}
A finite group $G$ has exactly two nonlinear irreducible characters if and only if one of the following
holds:
\nmrt
\tm{i}$G'$ is the union of three conjugacy classes of $G$ and for every $g\in G\setminus G'$, coset $gG'$ is
the conjugacy class of $G$ containing $g$,
\tm{ii}$G'$ is the union of two conjugacy classes of $G$ and there exists a normal subgroup $H$ of $G$ containing $G'$ such that
$H$ is the union of four conjugacy classes of $G$, $|H:G'|=2$ and for every $g\in G\setminus H$, cost $gG'$ is the conjugacy class of $G$
containing $g$.
 \enmrt
\end{corollary}
\begin{proof}
Let $(G,S)$ be the group association scheme of $G$. Clearly, if $G$ has exactly two nonlinear irreducible characters, then
$(G,S)$ also has two nonlinear irreducible characters and $|S|>3$.
Then it follows from Theorem \ref{20} that $G$ has exactly two nonlinear irreducible characters if and only if
one of the following holds:
\nmrt
\tm{1} $|\ooo(S)|=3$ and $S$ is the wreath product of association schemes $(G,S)_{g\ooo(S)}$ and
$(G/\!\!/\ooo(S),S/\!\!/\ooo(S))$ for $g\in G$,
\tm{2} $|\ooo(S)|=2$ and there exists a strongly normal closed subset $H$ of $S$ containing $\ooo(S)$ such that $|H|=4$, $|H:\ooo(S)|=2$
and $S$ is the wedge product of $(G,S)_{gH}$ and $(G/\!\!/\ooo(S),S/\!\!/\ooo(S))$ for $g\in G$.
\enmrt
We have statement $(1)$ if and only if $G'$ is the union of three conjugacy classes $C_0$, $C_1$ and $C_2$ and for every conjugacy
class $C_i\not\in \{C_0,C_1,C_2\}$,  $K_iK_1=K_1K_i=|C_1| K_i$ and
$K_iK_2=K_2K_i=|C_2| K_i$; see Lemma \ref{pedrs}. This shows that statement $(1)$ holds if and only if
$G'=C_0\cup C_1\cup C_2$ and for every conjugacy
class $C_i\not\in \{C_0,C_1,C_2\}$,
$C_1C_i \subset C_i$ and $C_2C_i \subset C_i$ and
so for every $x\in C_i$,  $xG'\subseteq C_i$.
Thus if $(i)$ occurs, then we clearly have statement $(1)$. Conversely, if statement $(1)$ holds, then
by using Lemma \ref{val} we have $|xG'|=|G'|=|C_i|$ and thus $xG'=C_i$. Hence $(i)$ holds.
\\

Statement $(2)$ holds if and only if $G'$ is the union of two conjugacy classes $C_0$ and $C_1$, $H$
is a normal subgroup of $G$ containing $G'$ with four conjugacy classes,  $|H:G'|=2$ and moreover,
for every conjugacy class $C_i $ of $G$ where $C_i\nsubseteq H$, $K_iK_1=K_1K_i=|C_1| K_i$.
The latter equality occurs if and only if for every conjugacy class $C_i $ of $G$ where $C_i\nsubseteq H$, $C_1C_i \subset C_i$ and
so  $gG'\subseteq C_i$ for every $g\in C_i$. Then we clearly have $(2)$ if $(ii)$ occurs. Conversely, suppose that $(2)$ holds.
Then $G'$ is the union of two conjugacy classes $C_0$ and $C_1$, $H$
is a normal subgroup of $G$ containing $G'$ with four conjugacy classes,  $|H:G'|=2$ and
for every conjugacy class $C_i $ of $G$ where $C_i\nsubseteq H$, $gG'\subseteq C_i$ for every $g\in C_i$.
Moreover, from Lemma \ref{val}, we have $|gG'|=|G'|=|C_i|$. So $gG'=C_i$ and statement $(ii)$ holds.

\end{proof}\hfill\bull

\begin{corollary}
A finite group $G$  has exactly two nonlinear irreducible characters if and only if one of the following
holds:
\nmrt
\tm{i} $G$ is an extra-special $3$-group,
\tm{ii}$G$ is a Frobenius group of order $\frac{p^n (p^{n}-1)}{2}$ with a cyclic Frobenius complement and
Frobenius kernel $G'$ which is an elementary abelian group of order $p^n$,
\tm{iii} $G$ is a Frobenius group with  Frobenius kernel $N$, an elementary abelian group of order $9$ such that $|G':N|=2$,
and with the Frobenius complement $Q_8$,
\tm{iv} $G$ is a $2$-group, $Z(G)$ is cyclic of order $4$ containing $G'$, and $G/Z(G)$ is elementary abelian.
 \enmrt

\end{corollary}
\begin{proof}
First assume that $G$ has exactly two nonlinear irreducible characters. Then one of the statements $(i)$ and $(ii)$ of
Corollary \ref{corj} holds.

Suppose that statement $(i)$ holds.  Let $G'$ be the union of conjugacy classes $C_0=\{1\}, C_1$ and $ C_2$.
We consider two cases.

First, suppose that $C_2=C^{-1}_1$. Then $G'$ is a $p$-group and so $G$ is solvable. Moreover, since for every $g\in G\setminus G'$,
coset $gG'$ is the conjugacy class of $G$ containing $g$, it follows that $G'$ is the unique minimal normal subgroup of $G$.
So it follows from \cite[Lemma 12.3]{is} that all nonlinear irreducible characters of $G$ have
equal degree $f$ and  one of the following holds:
\nmrt
\tm{a}$G$ is a $p$-group, $Z(G)$ is cyclic and $G/Z(G)$ is elementary abelian,
\tm{b}$G$ is a Frobenius group with an abelian Frobenius complement $H$ of
order $f$. Also, $G'$ is the Frobenius kernel and is an elementary abelian $p$-group.
\enmrt
If $(a)$ holds, then since $Z(G)\neq \{1\}$ and $G'\subseteq Z(G)$ we have $|G'|=3$ and $G'=Z(G)$. So $G$ is an extra-special
$3$-group.

Now suppose that $(b)$ holds. Then since $|G/G'|=f$ it follows that
$$f|G'|=|G|=\sum_{\chi\in \Irr(G)}\chi(1)^2=2f^2+|G/G'|=2f^2+f,$$
and so $f=\frac{|G'|-1}{2}$. Let $|G'|=p^n$. Then $G$ is a Frobenius group of order $\frac{p^n(p^n-1)}{2}$ with Frobenius kernel $G'$ and
a cyclic Frobenius complement of order  $\frac{p^n-1}{2}$.

Second, assume that $C^{-1}_1=C_1$ and $C^{-1}_2=C_2$. Clearly, in this case $G'$ cannot be abelian. Since $|G'|$ has at most two prime divisors it follows that $G'$ is solvable.
So $G$ is also solvable.  Moreover,  since $G'$ is solvable, $C_0\cup C_1$ or $C_0\cup C_2$
is a normal subgroup of $G$. Without loss in generality, assume that $L=C_0\cup C_1$ is a normal subgroup of $G$. Consider quotient
group $G/L$. Then $G/L$ has exactly one nonlinear irreducible character and it follows from Corollary \ref{con} that either
$G/L$ is an extra-special $2$-group, or $G/L$ is a doubly transitive Frobenius group with a cyclic complement and
Frobenius kernel $G'/L$ which is an elementary abelian $p$-group.

If $G/L$ is an extra-special $2$-group, then $|G'/L|=2$, $G'/L=Z(G/L)$ and $G/G'$ is a $2$-group. Clearly, $G$ is not a $2$-group and
\cite[Proposition 1]{camina} shows that $G$ is not a Frobenius group with Frobenius kernel $G'$. Let
$P$ be a $2$-Sylow subgroup of $G$. Since $G/L\simeq P$,  $P$ has class at most $2$. It follows  from
\cite[Theorem 5.1]{ch} that $G$ is a Frobenius group such that its Frobenius kernel has index $2$ in $G'$ with Frobenius complement
$Q_8$. Let $|L|=p^m$. Since $|C_1|=p^m-1$ divides $|G|=8p^m$ it follows that $p^m-1\mid 8$. So
$p=3$ and $m=2$. This is statement $(iii)$.

Now suppose that $G/L$ is a Frobenius group with Frobenius kernel $G'/L$ which is an elementary abelian $p$-group of
order $p^m$ and  with a cyclic Frobenius complement of order $p^m-1$.
If $G'$ is not a $p$-group, then it follows by the Frattini argument that $G=N_G(P)G'$ where $P$ is a $p$-Sylow subgroup of $G'$. So $|N_G(P)|=|G/G'|=p^m-1$. This is
a contradiction, since $|P|=p^m$.
Hence $G'$ must be a $p$-group. Since $(|G'|, |G/G'|)=1$, it follows from \cite[Proposition 1]{camina} that
$G$ is a Frobenius group with Frobenius kernel $G'$. Let $|G'|=p^n$. If $p\neq 2$, then  since the order of
Frobenius complement $G$ divides $p^n-1$,
it follows that the Frobenius complement has even order and so
Frobenius kernel $G'$ must be abelian. This is a contradiction. Thus we can assume that  $p=2$.
Since $G'/L$ is an elementary abelian $2$-group, it follows that $|G/L:C_{G/L}(gL)|= |G/L:G'/L|$ for every $g\in C_2$.
This implies that $|C_2|=|G:G'|$. Moreover, for every $g\in C_2$, the order of $gL$ is $2$ and so $g^2\in L$. So $|C_1|=|C_2|=|G:G'|$.
Hence $$|G'|=1+|C_1|+|C_2|=1+2|G/G'|=2^{m+1}-1.$$ This is a contradiction, since $|G'|=2^m$.
\\

Now assume that statement $(ii)$ holds. Let $G'$ be the union of two conjugacy classes $C_0=\{1\}$ and $C_1$. Then
$G'$ is the unique normal minimal subgroup of $G$. So $G'$ is a $p$-group and thus $G$ is solvable. It follows from
\cite[Lemma 12.3]{is} that all nonlinear irreducible characters of $G$ have
equal degree $f$ and either $(a)$ $G$ is a $p$-group, $Z(G)$ is cyclic and $G/Z(G)$ is elementary abelian; or
$(b)$ $G$ is a Frobenius group with an abelian Frobenius complement $H$ of
order $f$ and with Frobenius kernel $G'$ which is an elementary abelian $p$-group.

If $(a)$ holds, then  since $G'\subseteq Z(G)$ and $Z(G)\neq \{1\}$ it follows that $|G'|=2$, $Z(G)$ is the cyclic
group of order $4$ and $G/Z(G)$ is an elementary abelian $2$-group. This is statement $(iv)$.

Suppose that $(b)$ holds and let $|G'|=p^n$.
Since $|G/G'|=f$ it follows that
$$fp^n=|G|=\sum_{\chi\in \Irr(G)}\chi(1)^2=2f^2+f,$$
and so $f=\frac{p^n-1}{2}$. Then $G$ is a Frobenius group of order $\frac{p^n(p^n-1)}{2}$ with  Frobenius kernel $G'$ and
a cyclic Frobenius complement of order  $\frac{p^n-1}{2}$. This is statement $(ii)$.
\\

Conversely, if $G$ is an extra-special $3$-group, then it follows from \cite[Section 7]{hup} that $G$ has two nonlinear irreducible characters.
Moreover, if $G$ is a Frobenius group with a Frobenius complement $H$ and Frobenius kernel $F$, then $G$ has
$\frac{|\Irr(F)|-1}{|H|}$ nonlinear irreducible characters. So if $(ii)$ or $(iii)$ holds, then $G$ has two nonlinear irreducible
characters.  Finally, suppose that $(iv)$ holds. Then there are exactly two irreducible characters $\varphi, \theta\in \Irr(Z(G)) \setminus\Irr(Z(G)/G')$.
Clearly, $\varphi$ and $\theta$ are invariant in $G$ and there are exactly two irreducible characters $\chi, \psi\in \Irr(G)$ such
$[\varphi^G,\chi]\neq 0$ and $[\theta^G, \psi]\neq 0$.  So $\chi$ and $\psi$ are exactly two nonlinear irreducible characters of $G$.

\end{proof}\hfill\bull

The following theorem gives a class of association schemes with exactly two nonlinear irreducible characters.

\begin{theorem}\label{110}
Let $(X,S)$ be an association scheme such that $|\Irr(S)|>3$ and $S'$ is symmetric. Then $(X,S)$
contains exactly two nonlinear irreducible characters if and only if $(X,S)$ is a group-like scheme and
one of the following holds:
\nmrt
\tm{i}$|\widetilde{S'}|=3$ and $\widetilde{S}$ is the wreath product of association schemes $(X,\widetilde{S})_{x\widetilde{S'}}$
and $(X/\!\!/\widetilde{S'},\widetilde{S}/\!\!/\widetilde{S'})$ for $x\in X$,
\tm{ii}$|\widetilde{S'}|=2$ and there exists a strongly normal closed subset $H$ of ${S}$ such that
$|\widetilde{H}|=4$, $|H:S'|=2$ and
$\widetilde{S}$ is the wedge product of $(X,\widetilde{S})_{x\widetilde{H}}$ and $(X/\!\!/\widetilde{S'}, \widetilde{S}/\!\!/\widetilde{S'})$ for some $x\in X$.

\enmrt
\end{theorem}

\begin{proof}
First suppose  $(X,S)$ contains exactly two nonlinear irreducible characters  $\chi$ and $\psi$.
Since $\Irr(S/\!\!/S')$ acts on $\Irr(S)$ and
$\chi$ and $\psi$ are only nonlinear irreducible characters of $S$, it follows that either the orbits of $\chi$ and $\psi$ have
length $1$ or $\chi$ and $\psi$ lie in the same orbit.

First we assume that the orbits of $\chi$ and $\psi$ have length $1$. This implies that
$\chi_{S'}\neq \psi_{S'}$ and for every  $\zeta\in \Irr(S/\!\!/S')$,
$\chi\zeta=\chi$ and $\psi\zeta=\psi$. Since for every $s\in S\setminus S'$, there exists $\zeta\in \Irr(S/\!\!/S')$
such that $\zeta(\s_{s^{S'}})\neq 1$, from equalities  $\chi(\s_s)=\chi(\s_s)\zeta(\s_{s^{S'}})$ and
$\psi(\s_s)=\psi(\s_s)\zeta(\s_{s^{S'}})$ we have $\chi(\s_s)=\psi(\s_s)=0$. Now we prove that
$(X,S)$ is a group-like scheme. To do this, we show that for every $\lambda, \mu\in \Irr(S)$ and every $s,h\in S$, $\lambda\mu(\s_s\s_h)=\lambda\mu(\s_h\s_s)$
Then from Lemma \ref{glike}, $(X,S)$ is a group-like scheme. Let  $\lambda, \mu\in \Irr(S)$. Clearly, if
$\lambda$ or $\mu$ belongs to $\Irr(S/\!\!/S')$, then $\lambda\mu\in \Irr(S)$ and the result follows. So we can assume that
$\lambda, \mu\in \Irr(S)\setminus \Irr(S/\!\!/S')$. Then $\lambda, \mu\in \{\chi, \psi\}$. We show that $\lambda\mu(\s_s\s_h)=\lambda\mu(\s_h\s_s)$
for every $s,h\in S$. Clearly, if $s,h \in S'$, then $\s_s\s_h=\s_h\s_s$ and so $\lambda\mu(\s_s\s_h)=\lambda\mu(\s_h\s_s)$.
Moreover, if  either $s\in S'$ and $h\in S\setminus S'$ or $s,h\in S\setminus S'$ and $h\neq s^*$, then $$\s_s\s_h\displaystyle\sum_{k\in S\setminus S'}a_k\s_k,$$ and so
$$\lambda\mu(\s_s\s_h)=\displaystyle\sum_{k\in S \setminus S'}a_k\lambda\mu(\s_k)=\displaystyle\sum_{k\in S\setminus S'}a_k\frac{\lambda(\s_k)\mu(\s_h)}{n_k}=0$$
indeed, $\lambda(\s_s)=\mu(\s_s)=0$, for every $s\in S\setminus S'$. Similarly, $\lambda\mu(\s_h\s_s)=0$. So
$\lambda\mu(\s_h\s_s)=0=\lambda\mu(\s_s\s_h)$.
 Finally, we can assume that
$s,h\in S\setminus S'$ and $h=s^*$. Since $S'$ is symmetric it follows that
$\lambda_{ss^*k}=\lambda_{s^*sk^*}=\lambda_{s^*sk}$ and then
$$\s_s\s_{s^*}=\displaystyle\sum_{k\in S'}\lambda_{ss^*k}\s_k=\displaystyle\sum_{k\in S'}\lambda_{s^*sk}\s_{k}=\s_{s^*}\s_s.$$
So $\lambda\mu(\s_s\s_{s^*})=\lambda\mu(\s_{s^*}\s_s)$. Hence $(X,S)$ is a group-like scheme.

Now consider the association scheme $(X,\widetilde{S})$. Since $m_{\widetilde{\chi}}=\chi(1)m_\chi$ and $m_{\widetilde{\psi}}=\psi(1)m_\psi$
it follows that $\widetilde{\chi}$ and $\widetilde{\psi}$ are only nonlinear irreducible characters of $\widetilde{S}$; see \cite{hanaki}. Since
$\widetilde{\chi}(\s_{\widetilde{s}})\neq \widetilde{\psi}(\s_{\widetilde{s}})$, for every $\widetilde{s}\in \widetilde{S}'$,
part $(i)$ of Theorem \ref{20} shows that
$|\ooo(\widetilde{S})|=3$ and $\widetilde{S}$ is the wreath product of association schemes $(X,\widetilde{S})_{x\ooo(\widetilde{S})}$ and  $(X/\!\!/\ooo(\widetilde{S}),\widetilde{S}/\!\!/\ooo(\widetilde{S}))$ for $x\in X$. Since $\ooo(\widetilde{S})=\widetilde{S'}$,
we have the statement $(i)$.
\\

Now we assume that $\chi$ and $\psi$ lie in the same orbit. Then $\chi_{S'}=\psi_{S'}$ and
$\{\chi, \psi\}$ forms an orbit of length $2$. So $2$ divides $ |\Irr(S/\!\!/S')|$ and $S/\!\!/S'$ has a subgroup of order $2$.
Then there exists a strongly
normal closed subset $H$ of $S$ such that $S'\subseteq H$ and $|H:S'|=2$. Clearly, $\chi_H\neq \psi_H$. We show that $H$
is commutative. To do this, we prove that every irreducible character of $H$ has degree $1$. Suppose on contrary; that
there exists $\varphi\in \Irr(H)$ such that $\varphi(1)>1$.  Since $S'$ is commutative, it follows from Theorem \ref{pindex} that
$\varphi_{S'}$ is a sum of at most two distinct irreducible characters of $S'$ and $\varphi=\lambda^H$ for some irreducible
constituent $\lambda$ of $\varphi_{S'}$. On the other hand, since $\chi_{S'}=\psi_{S'}$, we conclude that $\lambda^S=a\chi+b\psi$
for some integers $a$ and $b$. Then $$a\chi+b\psi=\lambda^S=(\lambda^H)^S=\varphi^S.$$
This implies that $(\chi_H, \varphi)>0$ and $(\psi_H, \varphi)>0$. Then it follows from Theorem \ref{cli} that
$$\chi_{H}=\psi_{H}=e\displaystyle\sum^n_{i=1}\varphi_i,$$
where $\varphi_i, 1\leq i\leq n$, are $S/\!\!/{H}$-conjugate irreducible characters of $\varphi$. This is
a contradiction, indeed $\chi_H\neq \psi_H$. Hence $H$ is commutative.

Now by a similar way as above, one can see that for every $\lambda, \mu\in \Irr(S)$,
$\lambda\mu(\s_s\s_h)=\lambda\mu(\s_h\s_s)$ for every $s,h \in S$. Note that
$S'$ is symmetric, $H$ is commutative and $\chi(\s_s)=\psi(\s_s)=0$ for every $s\in S\setminus H$.
So it follows from Lemma \ref{glike} that $(X,S)$ is a group-like scheme.
Moreover, since  association scheme $(X,\widetilde{S})$ has exactly two nonlinear irreducible characters, it follows from
statement $(ii)$ of Theorem \ref{20} that $|\ooo(\widetilde{S})|=2$, $|\widetilde{H}|=4$ and
$\widetilde{S}$ is the wedge product of $(X,\widetilde{S})_{x\widetilde{H}}$ and
$(X/\!\!/\ooo(\widetilde{S}), \widetilde{S}/\!\!/\ooo(\widetilde{S}))$ for some $x\in X$.
This is statement $(ii)$. 

Conversely,  it follows from Theorem \ref{20}  that
$(X,\widetilde{S})$ has two nonlinear irreducible characters. Since $m_{\widetilde{\chi}}=\chi(1) m_\chi$, for every
$\chi\in \Irr(S)$ we conclude that $(X,S)$ has only two nonlinear irreducible characters. The proof is now complete.

\end{proof}\hfill\bull

\begin{remark}
The conditions $|\Irr(S)|>3$ and symmetry of $S'$ in
Theorem \ref{110} are  necessary conditions; see example below.
\end{remark}

\begin{example}
Let $(X,S)$ be the association scheme of order $21$, No. $19$ in \cite{hanakitable}, where
$S=\{s_0, \cdots, s_6\}$.
Then from \cite{hanakitable} the character table of the adjacency
algebra of $S$ is as follows.

$$\begin{array}{c|cccccc|c}
&\s_{s_0}&\s_{s_1}&\s_{s_2}&\s_{s_3}& \s_{s_4} &\s_{s_5} &m_\chi\\
\hline  \chi_1 &1&              2&                2&             4&4& 8& 1\\
 \chi_2 &1& -1&  -1& 1& 1&-1& 8\\
 \chi_3 &2& 1& 1& -2& -2& 0& 6\\
 \end{array}
$$
One can see that $(X,S)$ has two nonlinear irreducible characters, but it is not a group-like scheme
and the assertion of Theorem \ref{110} does not hold for
$(X,S)$. Moreover, if we consider the association scheme $H=S\wr K$,  where $K$ is
the trivial scheme of order $2$, then $H$ has order
$42$ and  the character table of $H$ is as follows; see \cite{irrwr}.
$$\begin{array}{c|ccccccc|c}
&\s_{s_0}&\s_{s_1}&\s_{s_2}&\s_{s_3}& \s_{s_4} &\s_{s_5}&\s_{s_6} &m_\chi\\
\hline  \chi_1 &1&              2&                2&             4&4& 8& 21& 1\\
\chi_2 &1&              2&                2&             4&4& 8& -21& 1\\
 \chi_2 &1& -1&  -1& 1& 1&-1& 0&16\\
 \chi_3 &2& 1& 1& -2& -2& 0& 0& 12\\

 \end{array}
$$
It is easy to see that  $H$ is not a group-like scheme. Although, $H$ contains exactly two nonlinear irreducible characters, but
$H'$ is not symmetric and so the conclusion of Theorem \ref{110} does not hold for $H$.

\end{example}

\end{document}